\documentclass[12pt]{article}

\usepackage{latexsym,amsfonts,amsmath,amssymb,epsfig,tabularx,amsthm,dsfont,mathrsfs}
\usepackage{graphicx}
\usepackage{hyperref}
\usepackage{enumerate}
\usepackage[dvipsnames]{xcolor}
\usepackage{titlesec}

\titleformat*{\section}{\large\bfseries}

\theoremstyle{definition}

\newtheorem{thm}{Theorem}[section]
\newtheorem{rem}[thm]{Remark}

\newtheorem{lem}[thm]{Lemma}
\newtheorem{prop}[thm]{Proposition}
\newtheorem{cor}[thm]{Corollary}
\newtheorem{example}[thm]{Example}

\renewcommand{\d}{{\rm d}}

\newcommand{\norm}[1]{\left\Vert #1 \right\Vert}

\newcommand{\N}{\mathbb{N}}
\newcommand{\E}{\mathbb{E}}
\newcommand{\R}{\mathbb{R}}

\DeclareMathOperator{\Var}{\text{Var}}

\DeclareMathOperator{\eps}{\varepsilon}
\DeclareMathOperator{\med}{\text{med}}

\title{Consistency of randomized\\ integration methods}

\author{Julian Hofstadler, Daniel Rudolf}

\date{\today}

\begin{document}

\maketitle

\begin{abstract}
We prove that a class of randomized integration methods, 
including averages based on $(t,d)$-sequences, Latin hypercube sampling, Frolov points as well as Cranley-Patterson rotations, consistently estimates expectations of integrable functions.
Consistency here refers to convergence in mean and/or convergence in probability of the estimator to the integral of interest.    
Moreover, we suggest median modified methods and show for integrands in $L^p$ with $p>1$ consistency in terms of almost sure convergence.
\end{abstract}


\textbf{Classification:} 65C05; 65D30.

	\section{Introduction}
In computational statistics and numerical analysis one of the major challenges is the development and investigation of estimators of expectations. A prototypical setting is the approximation of the integral
\begin{equation} 
	\label{eq: integral}
	I(f)=\int_{[0,1]^d} f(x)\, \d x 
\end{equation}
w.r.t. the Lebesgue measure for an integrable function $f\colon [0,1]^d\to \mathbb{R}$ with $d\in\mathbb{N}$. 
Given a (sufficiently rich) probability space $(\Omega,\mathcal{F},\mathbb{P})$, on which all random variables are defined, as well as
a measurable space $(G,\mathcal{G})$, a measurable function $g\colon G\to \mathbb{R}$ and a $G$-valued random variable $Y$ we can ask more generally for the computation of 
\begin{equation} \label{eq: expect}
	I(g) = \mathbb{E}[g(Y)].
\end{equation} 
For this goal (structured) Monte Carlo methods have been constructed.
We call $(S_n)_{n \in \N}$ a Monte Carlo method if for any $n \in \N$ and integrable $g$ we have that $S_ng$ is a real-valued random variable of the form $S_n g = \Phi_n(g(X_1), \dots, g(X_{N_n}))$, where $\Phi_n$ is a suitable (random) function, each $X_i$ is a $G$-valued random variable and $N_n$ is a $\N$-valued random variable that determines the number of allowed function evaluations for $g$.
Classical examples of such methods (exploiting different properties of $g$) for the approximation of $I(g)$ are for instance antithetic sampling, importance sampling or control variate based estimators, see e.g. \cite{muller2012monte,owen_2013_mcbook}.

We focus on the prototypical setting of the approximation of $I(f)$ from \eqref{eq: integral}, where $G=[0,1]^d$, and  consider structured methods in that context.
A minimal requirement for a reasonable proxy is that for sufficiently large and increasing $n$ it gets `close' to $I(f)$ whenever $f$ is integrable. We ask for the sequence of random
variables $(S_nf)_{n\in\mathbb{N}}$ that it at least, either, converges in probability, i.e., for
any $\varepsilon>0$ we have
\begin{equation}
	\label{eq: WLLN}
	\lim_{n \to \infty}\mathbb{P}\left[\vert S_nf-I(f) \vert >\eps\right]=0,
\end{equation}
or 
converges in (absolute) mean, i.e., 
\begin{equation}\label{equ:convergence_in_mean}
	\lim_{n\to\infty}\E \left[ \vert S_n f - I(f) \vert  \right] = 0.
\end{equation}
Note that also other formulations of convergence are feasible. For example, convergence in mean squared sense, that is, $\lim_{n\to \infty} \E [\vert S_n f - I(f)\vert^2  ]=0$. However, for $\E [\vert S_n f - I(f)\vert^2  ]$ to be finite one requires $f$ to be square integrable instead of just integrable, such that we do not further consider this concept. Another desirable asymptotic behavior, at least a more demanding property than \eqref{eq: WLLN}, is given by almost sure convergence, i.e.,  
\begin{equation}
	\label{eq: SLLN}
	\mathbb{P}\left[\lim_{n \to \infty} S_nf = I(f)\right]=1.
\end{equation}
For the classical Monte Carlo estimator $\frac{1}{n}\sum_{i=1}^n f(X_i)$, based on an iid sequence $(X_i)_{i \in \N}$ of random variables, uniformly distributed in $[0,1]^d$,
\eqref{eq: WLLN} is known as weak law of large numbers (WLLN) and \eqref{eq: SLLN} as strong law of large numbers (SLLN). Thus, we have convergence in probability and almost sure convergence for integrable $f$ of the standard estimator.
Motivated by this fact, we also sometimes say a (general) Monte Carlo method $(S_n)_{n \in \N}$ satisfies such a law of large numbers if the corresponding limit property (\eqref{eq: WLLN} or \eqref{eq: SLLN}) holds. 

Depending on integrability properties of the integrand $f$ we provide conditions for Monte Carlo methods $(S_n)_{n \in \N}$ to satisfy different types of consistency, namely convergence in probability and in mean as well as almost sure convergence. In fact, our results apply to more general methods which are introduced in Section~\ref{sec:wlln}.
For $p\geq 1$ let $L^p[0,1]^d$ be the set of all measurable functions $f \colon [0,1]^d \to \R$ with finite $\Vert f \Vert_p^p = \int_{[0,1]^d} \vert f(x)\vert^p \d x $.
We consider methods $(S_n)_{n \in \N}$
such that for any $n\in\mathbb{N}$, any $f, g \in L^1[0,1]^d$ and $\alpha , \beta \in \R$
we have 
\begin{itemize}
	\item[1)] linearity, i.e., $S_n (\alpha f + \beta g) = \alpha S_n f + \beta S_n g$ almost surely, 
	\item[2)] monotonicity, i.e., $\vert S_n f \vert \leq S_n \vert f \vert$ almost surely, and
	\item[3)] unbiasedness, i.e., $\E[S_n f] = I(f)$.
\end{itemize} 
We call a method $(S_n)_{n \in \N}$ linear, monotone or unbiased, respectively, if the corresponding property holds for any $n \in \N$.

Generically in Theorem~\ref{thm:convergence_L_1} we show that for this class of estimators  convergence in mean and in probability is equivalent for $f \in L^1[0,1]^d$.
Moreover, we consider estimators of \eqref{eq: integral} based on randomized $(t,d)$-sequences \cite{owen1995randomly}, Cranley-Patterson rotations \cite{cranley76randomization}, Latin hypercube sampling \cite{McKay1979} and randomized Frolov point sets \cite{krieg2017universal} 
and show that all of these methods are consistent regarding convergence in probability and mean whenever $f\in L^1[0,1]^d$.
After that, we argue for $f\in L^p[0,1]^d$ with $p>1$ how to get consistency in terms of almost sure convergence for median modified methods. In the latter scenario we follow the approach of \cite{owen2021strong}.

Now let us discuss how these results fit into the literature and whether they can be improved.
Given a Monte Carlo method $(S_n)_{n \in \N}$ and $f \in L^p[0,1]^d$ with $p\geq 1$ we define the expected absolute error as $\E[\vert S_n f -I(f) \vert ]$ and the
probability of failure $\varepsilon>0$ as
$\mathbb{P}[\vert S_n f - I(f) \vert > \eps]$.
Then, by virtue of
\cite[Theorem 2.3]{kunsch2019optimal} there are constants $\eps_0 >0$, $c>0$ and $n_0 \in \N$ with the following properties:
For any $n\geq n_0$, $\eps \in (0,\eps_0)$ and any arbitrary Monte Carlo method $(S_m)_{m \in \N}$ with $\E N_m \leq m$ we have
\begin{equation}\label{equ:lower_bound_prob_error}
	\sup_{\Vert f \Vert_1 \leq 1}\mathbb{P}\left[\vert S_nf - I(f) \vert >\eps\right] \geq c.
\end{equation}
As a consequence, see also \cite{He94,novak1988deterministic}, for the expected absolute error and $\eps = \eps_0/2$ we obtain 
\begin{equation}\label{equ:lower_bound_expt_error}
	\sup_{\Vert f \Vert_1 \leq 1}
	\E\left[ \vert S_nf - I(f) \vert \right] \geq \frac{c\cdot \eps_0}{2}.
\end{equation} 
By \eqref{equ:lower_bound_prob_error} and \eqref{equ:lower_bound_expt_error}, convergence in probability and convergence in mean cannot be improved to convergence uniformly 
for all $f \in L^1[0,1]^d$ with $\Vert f\Vert_1\leq 1$.
However, it is well known that the classical Monte Carlo estimator satisfies the SLLN for $f\in L^1[0,1]^d$ and that this implies the WLLN. 
Therefore, for fixed $f\in L^1[0,1]^d$ and a given Monte Carlo method (or a randomized integration method as defined in Section \ref{sec:wlln}) it is reasonable to ask for consistency regarding the convergence concepts mentioned above, even when the uniform probability of failure and the uniform expected absolute error do not decrease.

Motivated by Bayesian optimization \cite{balandat2020botorch} the recent work \cite{owen2021strong} provides a SLLN for $(t,d)$-sequences
randomized by a nested uniform scrambling for $f\in L^p[0,1]^d$ with $p>1$. The arguments there are based on the interpolation theorem of Riesz-Thorin as well as a subsequence technique. 
We consider a median of general estimators and show a SLLN without using the subsequence technique. However, we also require for the integrand $f\in L^p[0,1]^d$ with $p>1$. 
In contrast to that, our results about convergence in probability and in mean, see Theorem~\ref{thm:convergence_L_1}, Corollary~\ref{thm:Consistency_L_1_via_L_p} and Corollary~\ref{coro:consistency_by_variance}, require only $f\in L^1[0,1]^d$.
We apply these to different settings including randomized digital nets and Latin hypercube sampling, therefore extending the WLLNs of \cite{loh1996latin,owen2021strong}.
Moreover, the consequence of the convergence in mean of those estimators on $L^1[0,1]^d$ is to the best of our knowledge new.

Now we briefly outline the structure of the paper. {Section~\ref{sec:wlln} contains our results w.r.t. convergence in mean and in probability. 
	In particular, Theorem~\ref{thm:convergence_L_1}, Corollary~\ref{thm:Consistency_L_1_via_L_p} and Corollary~\ref{coro:consistency_by_variance} establish both types of convergence for $f\in L^1[0,1]^d$.}
We discuss estimators based on samples from randomized $(t,d)$-sequences, Latin hypercube sampling, Cranley-Patterson rotations and the randomized Frolov algorithm for which our results are applicable.
In Section \ref{sec:prob_amplification_and_SLLN} we describe how the median of independent realizations improves the 
probability of failure 
and how it can be used to derive a SLLN for $f\in L^p[0,1]^d$ with $p>1$.

\section{Convergence in mean and in probability}\label{sec:wlln}
The following theorem and its consequences provide our main tools for establishing convergence in probability and in mean for the aforementioned Monte Carlo methods.
However, our results hold for a larger class of estimators. Namely, we call a sequence of operators $(S_n)_{n \in \N}$  randomized integration method if $S_n \colon L^1[0,1]^d \to \mathcal{R}(\Omega)$ for any $n \in \N$, where $\mathcal{R}(\Omega)$ denotes the set of all real-valued random variables on $(\Omega, \mathcal{F}, \mathbb{P})$.
Clearly, any Monte Carlo method is a randomized integration method and,
moreover, the definitions of linearity, monotonicity and unbiasedness can be directly extended.

\begin{thm}\label{thm:convergence_L_1}
	Let $(S_n)_{n \in \N}$ be a randomized integration method that is linear, monotone and unbiased. Let
	$\mathcal{D} \subset L^1[0,1]^d$ be dense w.r.t. $\Vert \cdot \Vert_1$.
	Then the following statements are equivalent:
	\begin{enumerate}[(i)]
		\item \label{en: in_mean_L1}
		For any $f \in L^1[0,1]^d$ we have $\lim_{n\to \infty} \E[\vert S_nf - I(f)\vert ] = 0$.
		\item \label{en: in_mean_D}
		For any $\phi \in \mathcal{D}$ we have $\lim_{n\to \infty} \E[\vert S_n \phi - I(\phi)\vert ] = 0$. 
		\item \label{en: in_prob_D}
		For any $\phi \in \mathcal{D}$ we have that $(S_n\phi)_{n\in\mathbb{N}}$ converges in probability to $I(\phi)$.		
		\item \label{en: in_prob_L1}
		For any $f \in L^1[0,1]^d$ we have that $(S_nf)_{n\in\mathbb{N}}$ converges in probability to $I(f)$.
	\end{enumerate}
\end{thm}

\begin{proof}
	First, observe that \eqref{en: in_mean_L1} implies \eqref{en: in_mean_D} and applying Markov's inequality yields that \eqref{en: in_mean_D} implies \eqref{en: in_prob_D}. 
	
	We argue now that \eqref{en: in_prob_L1} is a consequence of \eqref{en: in_prob_D}. For this let $f \in L^1[0,1]^d$ and $\eps>0$ be arbitrary. 	
	Moreover, let $\delta \in (0,1) $ be arbitrary and choose $\phi \in \mathcal{D}$ such that $\norm{f-\phi}_{1} \leq \frac{\eps\delta}{6} < \frac{\eps}{3}$.
	We observe that 
	\begin{align*}
		\mathbb{P}\left[ \left\vert S_nf - I(f) \right\vert > \eps \right]
		\leq 
		\mathbb{P}\left[ \left\vert S_nf - S_n\phi \right\vert > \frac{\eps}{3} \right]
		+
		\mathbb{P}\left[ \left\vert S_n\phi - I(\phi) \right\vert > \frac{\eps}{3} \right] \\
		+
		\mathbb{P}\left[ \left\vert I(\phi) - I(f) \right\vert > \frac{\eps}{3} \right].
	\end{align*} 
	Combining Markov's inequality with linearity, monotonicity and unbiasedness of $S_n$ we deduce 
	\[
	\mathbb{P}\left[\vert S_nf - S_n \phi \vert > \frac{\eps}{3}\right] \leq 
	\frac{3}{\eps}\E\left[ \vert S_n (f-\phi) \vert \right]
	\leq 
	\frac{3}{\eps}\E\left[ S_n\vert f-\phi \vert \right]
	= \frac{3}{\eps}\norm{f-\phi}_1
	\leq \frac{\delta}{2}.
	\]
	The choice of $\phi$ implies that $\mathbb{P}\left[\vert I(f) - I(\phi) \vert > \eps/3\right] =0$, since $\vert I(f) - I(\phi) \vert \leq \norm{f-\phi}_1 < \eps/3$.
	By \eqref{en: in_prob_D} we have convergence in probability for $\phi \in \mathcal{D}$, hence there is some $n_0\in \N$ such that for any $n \geq n_0$ holds $\mathbb{P}\left[\vert S_n\phi - I(\phi) \vert > \eps/3\right] \leq \delta/2$.
	In conclusion 
	\[
	\mathbb{P}\left[ \left\vert S_n f -I(f) \right\vert > \eps \right] 
	\leq 
	\delta,
	\]
	for any $n \geq n_0$. By the fact that $\delta \in (0,1)$ was arbitrary we have that \eqref{en: in_prob_L1} follows. 
	
	It is left to show that \eqref{en: in_mean_L1} is a consequence of \eqref{en: in_prob_L1}.
	For this we first consider an arbitrary non-negative $f \in L^1[0,1]^d$.
	We aim to apply a well known characterization of convergence in mean w.r.t. convergence in probability that is (for the convenience of the reader) formulated in Lemma~\ref{lemma:WLLN_mean_convergence_equivalence}. 
	Monotonicity of $S_n$ implies that $S_n f \leq \vert S_n f \vert \leq S_n \vert f \vert  = S_n f$ almost surely. 
	Hence $S_n f = \vert S_n f \vert$ almost surely and consequently
	\[
	\E\left[ \vert S_n f\vert  \right] = \E\left[ S_n f \right] = I(f).
	\]
	This shows that $\left(\E\left[ \vert S_nf \vert \right]\right)_{n \in \N}$ is a constant sequence only taking the value $I(f) = \E[I(f)] = \E[\vert I(f)\vert ]$, which is clearly also its limit. 	
	By assumption we know that $(S_nf)_{n\in\mathbb{N}}$ converges in probability to $I(f)$ and therefore we can apply Lemma~\ref{lemma:WLLN_mean_convergence_equivalence} with $Y_n = S_nf $ and $Y=I(f)$. We obtain  
	\[
	\lim_{n \to \infty}
	\E[\vert S_n f -I(f) \vert ] =0.
	\]
	
	For the general case, i.e., when $f$ is not necessarily non-negative, we split $f$ into its positive and negative part, $f = f^+ - f^-$, with non-negative (integrable) functions $f^+, f^-$. 
	Linearity of $S_n$ and the integral together with the triangle inequality imply
	\[
	\E\left[ \vert S_n f -I(f) \vert \right] 
	\leq \E\left[ \vert S_n f^+ - I(f^+) \vert  \right]
	+ 
	\E\left[  \vert S_n f^- - I(f^-) \vert  \right].
	\]
	Since both, $f^+$ and $f^-$, are non-negative we have that
	\[
	\lim_{n \to \infty}\E[\vert S_n f^+ -I(f^+) \vert ] = 0
	\qquad\mbox{and}\qquad
	\lim_{n \to \infty}	\E[\vert S_n f^- -I(f^-) \vert ] = 0.
	\]
	Those limits give \eqref{en: in_mean_L1} and therefore the proof is finished. 
\end{proof}

\begin{rem}
	Theorem~\ref{thm:convergence_L_1} can be considered as generalization of a classical result appearing in the context of deterministic quadrature rules, cf. \cite[Theorem 3.1.2]{brass2011quadrature}.
	There, one has an uniform operator norm requirement of the quadrature rule which in our setting is somehow hidden in the assumptions about $S_n$.
	In particular, by monotonicity and unbiasedness of $S_n$ we have 
	\[
	\norm{S_n f}_{L^1(\Omega)} = \E \left[\vert S_n f \vert  \right]
	\leq 
	\mathbb{E}[S_n \vert f \vert ] = \norm{f}_{1}.
	\]
	Considering $f\equiv 1$, i.e., the function which only takes the value $1$, we obtain $\norm{S_n}_{L^1[0,1]^d \to L^1(\Omega)} = \sup_{\Vert f \Vert_1 \leq 1}\norm{S_nf}_{L^1(\Omega)} =1$ independent of $n \in \N$.
\end{rem}

\begin{rem}
	Linearity, monotonicity and unbiasedness are not sufficient to obtain consistency.   
	For instance, let $m \in \N$, let $Y_1, \dots, Y_m$ be independent and uniformly distributed in $[0,1]^d$ and set $X_j = Y_k$ if $j \; \mbox{mod} \;m =k$.  
	Then, consider 
	\[
	S_n f = \frac{1}{n}\sum_{j=1}^n f(X_j),
	\]
	where $f \in L^1[0,1]^d$. 
	It is clearly linear, monotone and unbiased, but it is not consistent.
	Indeed, if $A = \left[0, \frac{1}{2}\right] \times [0,1] \times \dots \times [0,1]$, $\eps \in \left(0, \frac{1}{2}\right)$ and $g \equiv \mathds{1}_A$,
	then for any $n \in \N$ satisfying $n \; \mbox{mod}\; m =0$ follows that
	$S_n g = \frac{1}{m}\sum_{j=1}^{m}g(Y_j)$ and consequently
	\[
	\mathbb{P}\left[ \vert S_n g - I(g) \vert > \eps \right] \geq 
	\mathbb{P}\left[ S_n g =1 \right] = \frac{1}{2^m},
	\]
	such that we have no consistency.
\end{rem}

The formerly stated theorem can be used to verify convergence in probability and in mean by exploiting 
the equivalences regarding the dense subset $\mathcal{D}$. 
To emphasize that, we add the following simple tool. 

\begin{cor}\label{thm:Consistency_L_1_via_L_p}
	Let $(S_n)_{n \in \N}$ be a linear, monotone and unbiased randomized integration method and let $p\geq 1$. 
	Assume that $(S_nf)_{n\in\mathbb{N}}$ converges in probability or in mean for arbitrary $f \in L^p[0,1]^d$ to $I(f)$. 
	Then, for any $f \in L^1[0,1]^d$ we have that $(S_nf)_{n\in\mathbb{N}}$ converges in probability and in mean to $I(f)$.
\end{cor}
\begin{proof}
	Denote by $C^\infty_c = C^\infty_c(0,1)^d$ the set of all functions that are infinitely often differentiable and whose support is a compact subset of $(0,1)^d$.
	We have that $C^\infty_c \subset L^p[0,1]^d$ for any $p\geq 1$, and it is a well known fact that $C^\infty_c$ is dense in $L^1[0,1]^d$ w.r.t. $\norm{\cdot }_1$.
	(In fact one has that $C^\infty_c$ is dense in $L^1(0,1)^d$. However, the open and closed unit cube only differ by a set of measure zero, so the denseness  easily extends to $L^1[0,1]^d$.)

	If $(S_nf)_{n\in\mathbb{N}}$ converges in probability or in mean for any $f \in L^p[0,1]^d$, then in particular we have this property for $\phi \in C_c^\infty$ and the statement of the corollary  is a consequence of Theorem \ref{thm:convergence_L_1}.
\end{proof}

Many structured methods are designed in such a way that they are unbiased and for $f \in L^2[0,1]^d$ their variance $\Var[ S_n f] = \E[\vert S_nf - I(f) \vert^2 ]$ decreases.
Assuming the underlying method is of this kind and satisfies the assumptions of the former corollary we can add the following tool for verifying the desired consistency for integrable functions.

\begin{cor}\label{coro:consistency_by_variance}
	Let $(S_n)_{n \in \N}$ be a  randomized integration method which is linear, monotone and unbiased. 
	Assume that for any $f \in L^2[0,1]^d$ we have 
	$
	\lim_{n \to \infty}\Var[S_nf] = 0.
	$
	Then, for $f \in L^1[0,1]^d$ we have that $(S_nf)_{n\in\mathbb{N}}$ converges in probability and mean to $I(f)$.
\end{cor}
\begin{proof}
	By Cauchy-Schwarz inequality we have for any $f\in L^2[0,1]^d$ that
	\[
	\mathbb{E}[\vert S_nf-I(f)\vert] \leq \Var[S_nf].
	\]
	By assumption this yields convergence in mean on $L^2[0,1]^d$, which implies by Corollary \ref{thm:Consistency_L_1_via_L_p} the claimed result.
\end{proof}

\begin{rem}\label{rem:wlln_general_domains}
	In Theorem~\ref{thm:convergence_L_1} and its consequences Corollary~\ref{thm:Consistency_L_1_via_L_p} and Corollary~\ref{coro:consistency_by_variance} we restricted ourselves to $G=[0,1]^d$, however, 
	this assumption can be relaxed to obtain convergence in probability and mean for estimating $I(g)$ as in \eqref{eq: expect}.
\end{rem}

Now we turn to examples. 
First, we consider estimators of the form
\begin{equation}
	\label{eq: est_repr}
	S_n(f) = \frac{1}{n} \sum_{y\in P_n} f(y), 
\end{equation}
where $S_n $ uses exactly $n$ function evaluations, that is, $P_n=\{ X_1^{(n)},\dots,X_n^{(n)} \}\subset [0,1]^d$ is a random point set, i.e., $X_i^{(n)}$ denotes a $[0,1]^d$-valued random variable for any $i=1,\dots,n$. Hence $(S_n)_{n \in \N}$ is a Monte Carlo method with $N_n \equiv~n$.
Obviously, $S_n$ is linear, monotone and whenever $P_n$ consists of uniformly distributed random variables in $[0,1]^d$ it is also unbiased.

The first two examples that we consider rely on exploiting the convergence on dense subsets of $L^1[0,1]^d$, see Theorem~\ref{thm:convergence_L_1} and Corollary~\ref{thm:Consistency_L_1_via_L_p}.

\begin{example}\label{ex:wlln_t_d_sequences}
	We consider randomized $(t,d)$-sequences as in \cite{owen1995randomly}. 
	The underlying deterministic points are $(t,d)$-sequences, which also prove to be useful for numerical integration (cf. \cite{dick2010digital}), and randomization is done, roughly speaking, by randomly permuting digits. 
	One obtains a sequence of random points, denoted by $(X_n)_{n \in \N}$, which consists of uniformly distributed points and satisfies crucial properties of a $(t,d)$-sequence with probability one, see \cite{owen1995randomly}.
	Setting $P_n = \{X_1, \dots, X_n\}$ and defining $S_n$ as specified in \eqref{eq: est_repr} we indeed have a linear, monotone and unbiased estimator.
	
	For $p>1$ the method $(S_n)_{n \in \N}$
	satisfies a SLLN, as recently shown in \cite{owen2021strong}, and consequently also a WLLN for $f \in L^p[0,1]^d$. 
	Thus all requirements to apply Corollary~\ref{thm:Consistency_L_1_via_L_p} are met and for $f \in L^1[0,1]^d$ it follows that 
	$(S_nf)_{n\in\mathbb{N}}$ converges in probability and in mean to $I(f)$.
\end{example}

\begin{example}\label{ex:cranley_patterson}
	A Cranley-Patterson rotation, cf. \cite{cranley76randomization}, randomizes deterministic points $\{a_1^{(n)}, \dots, a^{(n)}_n\}\subset [0,1]^d$ to $P_n=\{X^{(n)}_1, \dots, X^{(n)}_n\}$ by setting $X^{(n)}_i := a^{(n)}_i + U \mod 1$ (coordinate-wise) for $i=1,\dots,n$, where $U \sim \text{Unif}[0,1]^d$.
	The $X^{(n)}_i$'s are again uniformly distributed in $[0,1]^d$, hence $S_n$ as given in \eqref{eq: est_repr} is linear, monotone and unbiased.
	We refer to \cite[Chapter 17.3]{owen_2013_mcbook} for a more detailed discussion about this method.
	
	Employing \cite[Proposition 2.18]{dick2010digital}, which is a version of the Koksma-Hlawka inequality, and \cite[Inequality (17.9)]{owen_2013_mcbook} it follows that for $\phi \in C^\infty_c$ with probability one
	\[\vert S_n \phi - I(\phi) \vert \leq c_1 \cdot D(\{X^{(n)}_1, \dots, X^{(n)}_n\}) \leq c_2 \cdot D(\{a^{(n)}_1, \dots, a^{(n)}_n\}).
	\]
	Here $c_1,c_2$ are finite constants depending on $\phi,d$ only. Moreover, $C^\infty_c$ is as in the proof of Corollary~\ref{thm:Consistency_L_1_via_L_p} and $D(\cdot)$ denotes the discrepancy (of a point set), see for instance the book \cite{dick2010digital} for details.
	
	Assume now $\lim_{n \to \infty}D(\{a^{(n)}_1, \dots, a^{(n)}_n\}) = 0$, then we immediately obtain convergence in mean of $S_n$ for functions in $C^\infty_c$ and by virtue of Theorem~\ref{thm:convergence_L_1} it follows that $(S_nf)_{n\in\mathbb{N}}$ converges in probability and in mean for $f\in L^1[0,1]^d$.
	
	Under rather mild conditions, regarding the construction of a \textit{Hammersley point set} or a \textit{lattice point set} or a \textit{$(t,d)$-sequence} 
	the discrepancy for any of these point sets decreases to zero. 
	However, it might be necessary to switch to a subsequence,
	cf. \cite{dick2010digital}. 
	
	We want to emphasize here that no variance estimate for $S_n$ has been used.
	In particular obtaining useful variance estimates of $S_n$ for Cranley-Patterson rotated general point sets $\{a_1^{(n)}, \dots ,a_n^{(n)} \}$ with decreasing discrepancy seems to a be a difficult task, especially since we allow here that $a_j^{(n)} \neq a_j^{(k)}$ for $n \neq k$ and $1 \leq j \leq \min\{n,k\}$.
\end{example}

In the following two examples it is very convenient to exploit variance estimates, i.e., we apply Corollary~\ref{coro:consistency_by_variance}.

\begin{example}\label{ex:wlln_LHS}
	We study Latin hypercube sampling \cite{McKay1979}. 
	The corresponding algorithm produces uniformly distributed points satisfying stratification properties.
	We refer to  \cite{owen_2013_mcbook} for a detailed introduction and for more recent results to \cite{gnewuch2021discrepancy}.
	
	Using those points for $P_n$ within $S_n$ of \eqref{eq: est_repr} yields a linear, monotone and unbiased estimator. 
	From \cite{Stein}, see also \cite[Proposition~10.1]{owen_2013_mcbook}, for $f\in L^2[0,1]^d$ it follows that $\lim_{n \to \infty} \Var[S_n f]=0$ such that from Corollary~\ref{coro:consistency_by_variance} we obtain the statement of convergence in mean and probability of $(S_nf)_{n\in\mathbb{N}}$ to $I(f)$ for $f\in L^1[0,1]^d$.
	Consistency of Latin hypercube sampling is also very briefly mentioned in \cite{aistleitner2012central,packham2008latin}. 
	There it is meant in terms of convergence in probability for $f\in L^2[0,1]^d$ and one even has a SLLN in this case, cf. \cite{loh1996latin}.  
\end{example}

We add an example where $S_n$ does not take the form of \eqref{eq: est_repr}.

\begin{example}\label{ex:frolov_wlln}
	Randomized Frolov points, as studied in \cite{krieg2017universal} and later in \cite{ullrich2017monte}, 
	rely on a random shift and dilation of (deterministic) Frolov points (cf. \cite{ullrich2016upper} for an introductory paper), which themselves provide a powerful tool for numerical integration in Sobolev spaces, see
	\cite{ullrich2016role} for a survey.
	
	For technical reasons we extend any $f \colon [0,1]^d \to \R$ to a function defined on $\R^d$ by setting it zero outside the unit cube. 
	Hence for the integral of interest holds $I(f)=\int_{[0,1]^d}f(x) \d x = \int_{\R^d}f(x) \d x$.
	
	Here for any $n \in \N$ the estimator $S_n$ is given by
	\[
	S_nf = \frac{1}{\vert \det (A_n ) \vert} \sum_{y\in P_n} f(y),
	\]
	where $A_n\in \mathbb{R}^{d\times d}$ is a suitable random matrix and $P_n$ is the set of corresponding randomized Frolov points, which in particular depend on $A_n$. 
	We emphasize that $N_n = \mbox{card}(P_n)$, where $\mbox{card}(P_n)$ denotes the cardinality of $P_n$, is a random variable and
	\[
	\mathbb{E}\,[\vert\det(A_n) \vert] = \E\left[N_n\right]= n.
	\] 
	For details we refer to \cite[Section 2.1]{ullrich2017monte}.
	Obviously one has a linear and monotone estimator and \cite[Lemma 3]{krieg2017universal} guarantees that $S_n$ is also unbiased.

	For $f \in L^2[0,1]^d$, extended as above, we employ \cite[Theorem 1.1]{ullrich2017monte} and Plancherel's identity (see e.g. \cite[Theorem 2.2.14]{grafakos2014classical}) to deduce that 
	\begin{equation}\label{equ:Frolov_mean_squared}
		\Var\left[ S_n f \right]=
		\mathbb{E}\left[ \vert S_n f - I(f) \vert^2 \right]
		\leq 
		\frac{c \cdot\norm{f}_{2}^2}{n},
	\end{equation}
	where $c\in(0,\infty)$ does not depend on $f$ and $n$.
	Thus Corollary~\ref{coro:consistency_by_variance} is applicable and it follows that $(S_nf)_{n\in\mathbb{N}}$ converges in probability and in mean to $I(f)$ 		
	for $f \in L^1[0,1]^d$. 
	Note that in \cite{ullrich2017monte} the definition and analysis of this method is done for more general domains than $[0,1]^d$. 
	Therefore taking Remark \ref{rem:wlln_general_domains} into account this example can be generalized to different domains.
\end{example}

\section{{Almost sure convergence via median modification}
}\label{sec:prob_amplification_and_SLLN}

In this section we derive almost sure convergence statements for $f \in L^p[0,1]^d$, with $p>1$, by using median modified methods and combining techniques that were recently used in \cite{kunsch2019optimal,owen2021strong,kunsch2019solvable}. 
Let $(S_n)_{n \in \N}$ be  
a randomized integration method, $k\in\mathbb{N}$ an
odd number and 
$f\in L^p[0,1]^d$ for some $p>1$.
Throughout the whole section we write $S_{n,k}$ for the median of $k$ independent realizations of $S_n$, denoted by $S_n^{(1)}, \dots, S_n^{(k)} $, i.e.,
\[
S_{n,k}f = \med \{S_n^{(1)}f, \dots, S_n^{(k)}f\}.
\]  
We start with an auxiliary tool, see Proposition \ref{prop:SLLN_median_modified_methods}, which is then applied to obtain SLLNs for median modified methods relying on Latin hypercube sampling and randomized Frolov points.

Slightly modifying \cite[Proposition~2.1, see also (2.6)]{niemiro_pokarowski_2009} in \cite[Proposition~1]{kunsch2019optimal} the following lemma was proven.

\begin{lem}\label{lemma:median_trick}
	For $f\in L^1[0,1]^d$, $n\in \N$ and $\eps >0$ assume that $S_n$ satisfies 
	\[
	\mathbb{P}\left[ \vert S_nf - I(f) \vert >\eps  \right] \leq \alpha,
	\]
	for some $\alpha \geq 0$.
	Then, for any odd $k$ we have
	\[ 
	\mathbb{P}\left[\vert S_{n,k}f - I(f) \vert > \eps \right] \leq \alpha^{k/2} 2^k.
	\]
\end{lem}

\begin{rem}
	In \cite{kunsch2019optimal} Lemma~\ref{lemma:median_trick} is proven in the case where $S_n$ uses exactly $n$ function evaluations. 
	However, reviewing the proof of \cite[Proposition 2.1]{niemiro_pokarowski_2009} it follows that the number of used function values does not matter   
	and Lemma~\ref{lemma:median_trick} also holds for general randomized integration methods.
\end{rem}
\begin{rem}
	Recently for median modified randomized quasi-Monte Carlo methods, see \cite{pan2021super,pan2022super,goda2022construction}, worst case error bounds that hold with high probability have been proven. There, more regularity of the integrand is required and non-asymptotic statements similar as in \cite{kunsch2019optimal} are available.
\end{rem}

Combining Lemma~\ref{lemma:median_trick} and the Riesz-Thorin interpolation theorem, see Theorem~\ref{thm:Riesz_Thorin} below, we obtain the following result.

\begin{prop}\label{prop:SLLN_median_modified_methods}	
	Let $p \in (1,2)$, $k > 2/(p-1)$ be an odd number and  
	let $(S_n)_{n \in \N}$ be a linear, monotone and unbiased randomized integration method.
	Assume there is a constant $c \in (0, \infty)$ such that for any $n \in \N$ and any $f \in L^2[0,1]^d$ we have 
	\[
	\Var\left[ S_nf \right] \leq \frac{c \cdot \norm{f}_{2}^2}{n}.
	\]
	Then, the median of $k$ independent copies of $S_nf$ satisfies a SLLN for $f \in L^p[0,1]^d$, i.e., for any $f \in L^p[0,1]^d$ the sequence of random variables $(S_{n,k}f)_{n\in\mathbb{N}}$ converges almost surely to $I(f)$.
\end{prop}

\begin{proof}
	Define the measure spaces $\mathcal{A}_1 = ([0,1]^d, \mathcal{B}([0,1]^d), \lambda_d)$ and $\mathcal{A}_2 = (\Omega, \mathcal{F}, \mathbb{P})$, where $\lambda_d$ denotes the $d$-dimensional Lebesgue measure.
	Set $p_1=1$, $p_2=2$ and $\theta = 2-\frac{2}{p}$. For $i \in \{1,2\}$ we denote by $L^{p_i}(\mathcal{A}_1)$ and $L^{p_i}(\mathcal{A}_2)$ the corresponding Lebesgue spaces.
	Then, the operators $T_n \colon L^{p_i}(\mathcal{A}_1) \to L^{p_i}(\mathcal{A}_2)$ with $T_nf = S_nf-I(f)$ are linear and bounded. 
	An application of the Riesz-Thorin interpolation theorem, see Theorem~\ref{thm:Riesz_Thorin} below, yields
	\[
	\norm{T_n}_{L^p(\mathcal{A}_1) \to L^p(\mathcal{A}_2)} 
	\leq 
	\frac{2^{1-\theta} c^{\theta/2}}{n^{\theta/2}} 
	= 
	\frac{c_p}{n^{1-\frac{1}{p}}},
	\]
	with $c_p=2^{2/p-1} c^{1-1/p}$.
	Here we used, with the same notation as in Theorem~\ref{thm:Riesz_Thorin}, that $M_1 \leq 2$ since $S_n$ is unbiased and $M_2 \leq (c/n)^{1/2}$ according to the assumption about the variance of $S_n$.
	Therefore, by Markov's inequality, for $f \in L^p[0,1]^d$ and $\eps>0$ we get
	\begin{align}\label{equ:p_mean_error_bound}
		\mathbb{P}\left[ \vert S_nf - I(f) \vert > \eps \right] 
		&\leq
		\frac{1}{\eps^p} \E \left[\vert T_n f\vert^p\right] \\
		&\leq \nonumber 
		\left( \frac{1}{\eps} \norm{T_n}_{L^p(\mathcal{A}_1) \to L^p(\mathcal{A}_2)} \norm{f}_p\right)^p  \\
		&\leq \nonumber
		\left( \frac{c_p \norm{f}_p}{\eps}\right)^p \frac{1}{n^{p-1}}.
	\end{align}
	Applying Lemma \ref{lemma:median_trick} we obtain
	\[
	\mathbb{P}\left[ \vert S_{n,k}f-I(f) \vert > \eps \right]
	\leq 
	2^k \left( \frac{c_p \norm{f}_p}{\eps}\right)^{pk/2} \frac{1}{n^{(p-1)k/2}}.
	\] 
	By the fact that $k > 2/(p-1)$ we have $\sum_{n=1}^{\infty} n^{-(p-1)k/2} < \infty$ and the result follows by the lemma of Borel-Cantelli.
\end{proof}

\begin{rem}\label{rem:SLLN_subsequence}
	Note that without using the median of several independent copies of $S_n$ it is still possible to find SLLNs for $f \in L^p[0,1]^d$ with $1<p<2$ by considering $(S_{n_j})_{j\in \N}$ for a suitably chosen (sub)sequence $(n_j)_{j \in \N}$.
	Namely, for a method satisfying the same assumptions as in Proposition \ref{prop:SLLN_median_modified_methods} we obtain by \eqref{equ:p_mean_error_bound} that for $\eps >0$ and any $n_j\in\mathbb{N}$ holds
	\[
	\mathbb{P}\left[ \vert S_{n_j} f -I(f) \vert > \eps  \right] \leq \left( \frac{c_p \norm{f}_p}{\eps}\right)^p \frac{1}{n_j^{p-1}}, 
	\]
	where $c_p$ is as in the proof of Proposition \ref{prop:SLLN_median_modified_methods}.
	Thus by choosing the sequence $(n_j)_{j \in \N}$ in such a way that $\sum_{j=1}^{\infty} {n_j^{-p+1}} < \infty$ it follows from the lemma of Borel-Cantelli that $(S_{n_j}f)_{j\in\mathbb{N}}$ converges almost surely to $I(f)$.
\end{rem}

We illustrate the applicability of the former proposition for Latin hypercube sampling and randomized Frolov points. 

\begin{example}
	We consider $S_n$ with $P_n$ based on Latin hypercube sampling as in Example~\ref{ex:wlln_LHS}.
	We argued there already that $S_n$ is linear, monotone and unbiased. 
	By \cite[Corollary 17.1]{owen_2013_mcbook} a variance bound for $S_n$ is established. 
	It allows us to use Proposition \ref{prop:SLLN_median_modified_methods} and hence for the median $S_{n,k}$, with $k > 2/(p-1)$ odd, we obtain a SLLN. 
	
	Furthermore, it is also possible to deduce SLLNs for suitably chosen subsequences, for instance as suggested in Remark \ref{rem:SLLN_subsequence}.
\end{example}

\begin{example}
	We continue with $S_n$ from Example~\ref{ex:frolov_wlln} based on randomized Frolov points, where it was already shown that $S_n$ is linear, monotone and unbiased.
	By \eqref{equ:Frolov_mean_squared} we have a variance estimate as required in Proposition~\ref{prop:SLLN_median_modified_methods}, 
	resulting in a SLLN for $f \in L^p[0,1]^d$, for the median $S_{n,k}$ of $k > 2/(p-1)$ independent copies ($k$ odd), or for $S_n$ by passing to a suitable subsequence (see Remark~\ref{rem:SLLN_subsequence}).
\end{example}

\begin{appendix}
	\section{Auxiliary results}
	\renewcommand{\thesection}{\Alph{section}}
	\setcounter{section}{1}		
	For the convenience of the reader we add here a standard result about the characterization of convergence in mean that can be, for example, found in \cite[Theorem 5.12]{Kallenberg2021foundations}. 
	\begin{lem}\label{lemma:WLLN_mean_convergence_equivalence}
		Let $(Y_n)_{n\in\mathbb{N}}$ be a sequence of real-valued integrable random variables as well as $Y$ being an integrable real-valued random variable. 
		Then the following statements are equivalent:
		\begin{enumerate}[(i)]
			\item We have $\lim_{n \to \infty}\E\left[ \vert Y_n - Y\vert  \right] = 0$.
			\item The sequence of random variables $(Y_n)_{n\in\mathbb{N}}$ converges in probability to $Y$ and $\lim_{n\to\infty} \E[\vert Y_n\vert ] = \E[\vert Y \vert ]$.
		\end{enumerate}
	\end{lem}
	
	In the proof of Proposition~\ref{prop:SLLN_median_modified_methods} we used the following version of the Riesz-Thorin interpolation theorem that can be found in \cite[Theorem 1.3.4]{grafakos2014classical}.
	
	\begin{thm}[Riesz-Thorin]\label{thm:Riesz_Thorin}
		Let $\mathcal{A}_1$ and $\mathcal{A}_2$ be measure spaces. 
		Let $T$ be a linear operator from $L^{p_1}(\mathcal{A}_1)$ to $L^{p_1}(\mathcal{A}_2)$ and also a linear operator from 
		$L^{p_2}(\mathcal{A}_1)$ to $L^{p_2}(\mathcal{A}_2)$, for some $1 \leq p_1 \leq p_2 \leq \infty $.
		Assume that 
		\[
		\norm{T}_{L^{p_1}(\mathcal{A}_1) \to L^{p_1}(\mathcal{A}_2)} \leq M_1  
		\qquad\text{and}\qquad
		\norm{T}_{L^{p_2}(\mathcal{A}_1) \to L^{p_2}(\mathcal{A}_2)} \leq M_2.  
		\]
		For $\theta \in (0,1)$ define $p$ via $\frac{1}{p} = \frac{1-\theta}{p_1}+\frac{\theta}{p_2}$.
		Then, $T$ is also a linear operator from $L^p(\mathcal{A}_1)$ to $L^p(\mathcal{A}_2)$ satisfying
		\[
		\norm{T}_{L^p(\mathcal{A}_1) \to L^p(\mathcal{A}_2)} \leq M_1^{1-\theta}M_2^{\theta}.
		\]
	\end{thm} 
	
\end{appendix}

\section*{Acknowledgements}
Julian Hofstadler gratefully acknowledges support of the DFG within project 432680300 -- SFB 1456 subproject B02.
The authors thank Michael Gnewuch, David Krieg, Erich Novak, Art Owen, Mario Ullrich, Marcin Wnuk  and the anonymous referees for valuable comments that helped to improve the paper significantly.

\newcommand{\etalchar}[1]{$^{#1}$}

\noindent	
\textbf{Information about the authors:\\[-1.5ex]} 

\noindent
Julian Hofstadler, Universit\"at Passau, Faculty for Computer Science and Mathematics, Innstraße 33, 94032 Passau, Germany. 

\noindent
\textit{Email:} julian.hofstadler@uni-passau.de   

\medskip
\noindent
Daniel Rudolf, Universit\"at Passau, Faculty for Computer Science and Mathematics, Innstraße 33, 94032 Passau, Germany. 

\noindent
\textit{Email:} daniel.rudolf@uni-passau.de   

\end{document}